\documentclass[12pt,reqno]{amsart}

\usepackage{amsmath, amsthm, amssymb,verbatim,fullpage,longtable,multicol,multirow}
\usepackage{graphicx}
\usepackage{hyperref}
\usepackage[table]{xcolor}
\usepackage[T1]{fontenc}
\usepackage{appendix}
\usepackage[utf8]{inputenc}
\usepackage{tikz}
\usepackage{braket}
\usepackage{mathtools}
  \usepackage{subfig}
        \usepackage{multicol}
        \usepackage{wrapfig}
        \usepackage{float}
\usepackage{color}

\usepackage{enumitem}
\usepackage{caption}
\usepackage{hyperref}
\usepackage{url}
\usepackage{breakurl}
\newcommand{\bburl}[1]{\textcolor{blue}{\url{#1}}}

\newtheorem{theorem}{Theorem}[section]
\newtheorem{lemma}[theorem]{Lemma}
\newtheorem{proposition}[theorem]{Proposition}
\newtheorem{corollary}[theorem]{Corollary}
\newtheorem{definition}[theorem]{Definition}

\newtheorem{example}[theorem]{Example}

\newcommand{\cm}{c_{{\rm mean}}}
\newcommand{\cv}{c_{{\rm var}}}

\definecolor{lightgray}{gray}{0.7}
\definecolor{midgray}{gray}{.9}
\newcommand\be{\begin{equation}}
\newcommand\ee{\end{equation}}
\newcommand\bea{\begin{eqnarray}}
\newcommand\eea{\end{eqnarray}}
\newcommand\bi{\begin{itemize}}
\newcommand\ei{\end{itemize}}
\newcommand\ben{\begin{enumerate}}
\newcommand\een{\end{enumerate}}

\numberwithin{equation}{section}

\begin{document}

\title{A Generalization of Zeckendorf's Theorem via Circumscribed $m$-gons}

\author{Robert Dorward}
\email{\textcolor{blue}{\href{mailto:rdorward@oberlin.edu}{rdorward@oberlin.edu}}}
\address{Department of Mathematics, Oberlin College, Oberlin, OH 44074}

\author{Pari L. Ford}
\email{\textcolor{blue}{\href{mailto:fordpl@bethanylb.edu}{fordpl@bethanylb.edu}}}
\address{Department of Mathematics and Physics, Bethany College, Lindsborg, KS 67456}

\author{Eva Fourakis}
\email{\textcolor{blue}{\href{mailto:erf1@williams.edu}{erf1@williams.edu}}}
\address{Department of Mathematics and Statistics, Williams College, Williamstown, MA 01267}

\author{Pamela E. Harris}
\email{\textcolor{blue}{\href{mailto:pamela.harris@usma.edu}{pamela.harris@usma.edu}}}
\address{Department of Mathematical Sciences, United States Military Academy, West Point, NY 10996}

\author{Steven J. Miller}
\email{\textcolor{blue}{\href{mailto:sjm1@williams.edu, Steven.Miller.MC.96@aya.yale.edu}{sjm1@williams.edu, Steven.Miller.MC.96@aya.yale.edu}}}
\address{Department of Mathematics and Statistics, Williams College, Williamstown, MA 01267}

\author{Eyvi Palsson}
\email{\textcolor{blue}{\href{mailto:eap2@williams.edu}{eap2@williams.edu}}}
\address{Department of Mathematics and Statistics, Williams College, Williamstown, MA 01267}

\author{Hannah Paugh}
\email{\textcolor{blue}{\href{mailto:hannah.paugh@usma.edu}{hannah.paugh@usma.edu}}}
\address{Department of Mathematical Sciences, United States Military Academy, West Point, NY 10996 }

\date{\today}



\thanks{The authors thank CPT Joseph Pedersen for his programming assistance in the creation of Figure~\ref{gaussiangraphs}. The authors would also like to thank the AIM REUF program, the SMALL REU and Williams College for all of their support and continued funding, which made this work and collaboration possible. Moreover, we thank West Point, in particular the Center for Leadership and Diversity in STEM, for travel funding in support of this work and its' dissemination. The first and third named authors were supported by NSF Grant DMS1347804, and the fifth named author by NSF Grant DMS1265673. This research was performed while the fourth named author held a National Research Council Research Associateship Award at USMA/ARL}


\maketitle

\begin{abstract}
Zeckendorf's theorem states that every positive integer can be uniquely decomposed as a sum of nonconsecutive Fibonacci numbers, where the Fibonacci numbers satisfy $F_n=F_{n-1}+F_{n-2}$ for $n\geq 3$, $F_1=1$ and $F_2=2$. The distribution of the number of summands in such decomposition converges to a Gaussian, the gaps between summands converges to geometric decay, and the distribution of the longest gap is similar to that of the longest run of heads in a biased coin; these results also hold more generally, though for technical reasons previous work needed to assume the coefficients in the recurrence relation are non-negative and the first term is positive.

We extend these results by creating an infinite family of integer sequences called the $m$-gonal sequences arising from a geometric construction using circumscribed $m$-gons. They satisfy a recurrence where the first $m+1$ leading terms vanish, and thus cannot be handled by existing techniques. We provide a notion of a legal decomposition, and prove that the decompositions exist and are unique. We then examine the distribution of the number of summands used in the decompositions and prove that it displays Gaussian behavior. There is geometric decay in the distribution of gaps, both for gaps taken from all integers in an interval and almost surely in distribution for the individual gap measures associated to each integer in the interval. We end by proving that the distribution of the longest gap between summands is strongly concentrated about its mean, behaving similarly as in the longest run of heads in tosses of a coin.
\end{abstract}

\section{Introduction} \label{sec:intro}

The Fibonacci numbers are a heavily studied sequence which arise in many different ways and places. By defining them by $F_1 = 1$, $F_2 = 2$ and $F_{n+1} = F_n + F_{n-1}$, we have the remarkable property that every positive integer can be uniquely written as a sum of non-consecutive Fibonacci numbers; further, this property is equivalent to the Fibonaccis (i.e., if $\{a_n\}$ is a sequence of numbers such that every integer can be written uniquely as a sum of non-adjacent terms in the sequence, then $a_n = F_n$). Zeckendorf proved this in 1939, though he did not publish this result until 1972 \cite{Ze}.

In recent years many have studied generalizations to Zeckendorf's theorem by exploring different notions of decompositions and the properties of the associated sequences, see among others \cite{Al,Day,DDKMMV,DDKMV,DG,FGNPT,GT,GTNP,Ke,Len,MW1,MW2,Ste1,Ste2}. Despite the vast literature in this area, the majority of the research on generalized Zeckendorf decompositions have involved sequences with \textit{positive linear recurrences}. Positive linear recurrence sequences $\{G_n\}$ satisfy a linear recurrence relation  where the coefficients are non-negative with the first and last term coefficients being positive\footnote{Thus $G_{n+1} = c_1 G_n + \cdots + c_L G_{n-(L-1)}$ with $c_1 c_L > 0$ and $c_i \ge 0$.}.


There has been little research which considers cases where the leading coefficient in the recurrence is zero; one such case is found in \cite{CFHMN1}.  They studied what they call the Kentucky Sequence, which is defined by the recurrence relation $H_{n+1}=H_{n-1}+2H_{n-3}$, $H_i=i$ for $i \leq 4$. While the behavior there is similar to the positive linear recurrences, there are sequences with very different behavior. One such is the Fibonacci Quilt, which arises from creating a decomposition rule from the Fibonacci spiral\footnote{Let $f_n = F_{n-1}$, the standard definition of the Fibonacci numbers. Then the plane can be tiled in a spiral where the dimensions of the $n$\textsuperscript{th} square is $f_n \times f_n$; we declare a decomposition legal if no two summands used share an edge.} (see \cite{CFHMN2,CFHMN3}), where the number of decompositions is not unique but in fact grows exponentially. This leads to the major motivation of this paper (as well as the motivation for the three papers just mentioned): how important is the assumption that the leading term be positive? The work mentioned above shows that it is not just a technically convenient assumption; markedly different behavior can emerge. Our goal is to try and determine when we have each type of behavior, and thus the purpose of this paper is to explore infinitely many recurrences with leading term absent and see the effect that has on the properties of the decompositions.

Specifically, we consider an infinite family of integer sequences called the \emph{$m$-gonal sequences}, where $m\geq 3$. These sequences arise from a geometric construction using circumscribed $m$-gons, and after defining them below we state our results.

\subsection{Definition of $m$-gonal Sequence}

One interpretation of the Zeckendorf's theorem, which state that every positive integer can be written uniquely as a sum of non-consecutive Fibonacci numbers, is that we have infinitely many bins with just one number per bin, and if we choose a bin to contribute a summand to a number's decomposition then we cannot choose a summand from an adjacent bin. We can generalize to bins with more elements,  as well as disallowing two bins to be used if they are within a given distance (see \cite{CFHMN1,CFHMN2,CFHMN3}). The $m$-gonal sequences are similar to these constructions, but have a two-dimensional structure arising from circumscribing $m$-gons about one central $m$-gon.

Briefly we view the decomposition rule corresponding to the $m$-gonal sequence, for $m \geq 1$, by saying the sequence is partitioned into bins $b_i$ of length $|b_i|$, where $|b_0|=1$ and $|b_i|=m$ for all $i\geq 1$.  A valid decomposition has no two summands being elements from the same bin. We refer to this decomposition as a \emph{legal $m$-gonal decomposition of a positive integer $z$}. We now give details and examples of this construction.

The $m$-gonal sequence was initially constructed by circumscribing $m$-gons.  For $m\geq 3$ we let $M_0$ denote a regular $m$-gon. Circumscribe the $m$-gon $M_1$ onto $M_0$ such that the vertices of $M_0$ bisect the edges of $M_1$. Note that this adds $m$ faces to the resulting figure. We continue this process indefinitely, where we circumscribe the $m$-gon $M_{i}$ onto $M_{i-1}$, such that the vertices of $M_{i-1}$ bisect the edges of $M_{i}$. At each step we have added an additional $m$ faces to the resulting figure. We depict these initial iterations in Figure \ref{fig:polys}.  Let $\mathcal{M}$ denote all of the faces created  through the process of circumscribing $m$-gons. Then
\begin{equation}\mathcal{M}\ = \ \{f_0\}\cup\left(\displaystyle\bigcup_{i=1}^{\infty}\{f({i,1}),f({i,2}),\ldots,f({i,m})\}\right),\end{equation}
where $f_0$ is the face of the $m$-gon $M_0$ and for $i\geq 1$, $f({i,1}),f({i,2}),\ldots,f({i,m})$ are the faces added to $\mathcal{M}$ when $M_{i}$ was circumscribed onto $M_{i-1}$.

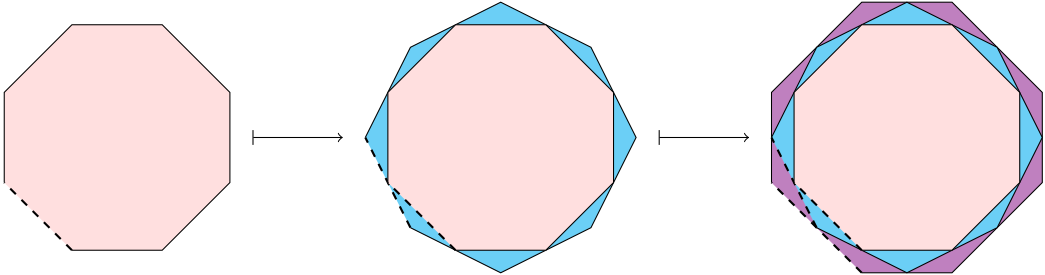
\begin{figure}[h!]
\centering
\begin{tikzpicture} [scale=.60]
\draw [fill=pink!50] (0,1.5) -- (0,3.5) -- (1.5,5) -- (3.5,5) -- (5,3.5) -- (5,1.5) -- (3.5,0) -- (1.5,0);
\draw [dashed, thick] [fill=pink!50] (1.5,0) -- (0,1.5);
\draw [|->] (5.5,2.5) -- (7.5,2.5) ;
\draw [fill=cyan!50](8,2.5) -- (9,4.5) -- (11,5.5) -- (13,4.5) -- (14,2.5) --(13,.5) -- (11,-.5) -- (9,.5);
\draw [dashed, thick] [fill=cyan!50] (9,.5) -- (8,2.5);
\draw [fill=pink!50] (8.5,1.5) -- (8.5,3.5) -- (10,5) -- (12,5) -- (13.5,3.5) -- (13.5,1.5) -- (12,0) -- (10,0);
\draw [dashed, thick] [fill=pink!50] (10,0) -- (8.5,1.5);
\draw [|->] (14.5,2.5) -- (16.5,2.5) ;
\draw [fill=violet!50](17,1.5) -- (17,3.5) -- (19,5.5) -- (21,5.5) -- (23,3.5) -- (23,1.5) -- (21,-.5) -- (19,-.5);
\draw [dashed, thick] [fill=lime!50] (19,-.5) -- (17,1.5);
\draw [fill=cyan!50](17,2.5) -- (18,4.5) -- (20,5.5) -- (22,4.5) -- (23,2.5) --(22,.5) -- (20,-.5) -- (18,.5);
\draw [dashed, thick] [fill=cyan!50] (18,.5) -- (17,2.5);
\draw [fill=pink!50] (17.5,1.5) -- (17.5,3.5) -- (19,5) -- (21,5) -- (22.5,3.5) -- (22.5,1.5) -- (21,0) -- (19,0);
\draw [dashed, thick] [fill=pink!50] (19,0) -- (17.5,1.5);
\end{tikzpicture}
\caption{Circumscribed $m$-gons}
\label{fig:polys}
\end{figure}

Fix an integer $m \geq 3.$  Suppose $\{a_n\}_{n=0}^{\infty}$ is an increasing sequence of positive integers.  We define the following ordered lists, which we refer to as bins, $b_0=[a_0]$ and for $i\geq 1$, $b_i=[a_{m(i-1)+1}$, $a_{m(i-1)+2}$, $\dots$, $a_{mi}]$. Note that for all $i\geq 1$, $b_i$ has size $m$ and $b_0$ has size one. The integers in bin $b_i$ will correspond directly with the integers which we place on the faces added to $\mathcal{M}$ when $M_{i}$ was circumscribed onto $M_{i-1}$.  With the elements of our sequence partitioned into bins, we define a \emph{legal $m$-gonal decomposition} of any positive integer $z$.  If we have
\begin{equation}z \ = \ a_{\ell_t}+a_{\ell_{t-1}}+\cdots+ a_{\ell_2}+a_{\ell_1},\end{equation}
where $\ell_1<\ell_2<\cdots<\ell_t$ and $\{a_{\ell_j}, a_{\ell_{j+1}}\}\not\subset b_i$ for any $i\geq 0$ and $1\leq j\leq t-1$, then we call this a legal $m$-gonal decomposition of $z$. Namely, a legal $m$-gonal decomposition cannot use more than one summand from the same bin.  And with the generalized construction of the sequence by partitioning the members into bins rather than relying solely on the 2-dimensional circumscribed polygons, we make it a formal definition for $m \geq 1$.

\begin{definition}\label{mDefi}
Let an increasing sequence of positive integers $\{a_n\}_{n=0}^\infty$ be given and partition the elements into ordered lists that we call bins  \begin{equation}b_k \ := \ [a_{m(k-1)+1}, a_{m(k-1)+2},\ldots, a_{mk}]\end{equation} for $m \geq 1,$ $k\geq 1,$ and $b_0=[a_0]$.  We declare a decomposition of an integer \begin{equation} z\ = \ a_{\ell_t} + a_{\ell_{t-1}} + \cdots + a_{\ell_1}\end{equation}  where $\ell_1 < \ell_2  <  \cdots < \ell_t$ and $\{a_{\ell_j}, a_{\ell_{j+1}}\} \not\subset b_i$ for any $i,j$ to be a \emph{legal $m$-gonal decomposition}.
\end{definition}

The following definition details the construction of the $m$-gonal sequence, which is the focus of this paper.

\begin{definition}
For $m\geq 1$, an increasing sequence of positive integers $\{a_n\}_{n=0}^\infty$ is called an \emph{$m$-gonal sequence} if every $a_i$ ($i \geq 0$) is the smallest positive integer that does not have a legal $m$-gonal decomposition using the elements $\{a_0, a_1, \dots, a_{i-1}\}.$
\end{definition}

\begin{example}
For $m =1,$ all the bins have size 1 and the 1-gonal sequence $\{a_i\}_{i=0}^{\infty}$ is defined by $a_i = 2^i$. This is equivalent to writing an integer in binary. When $m=2$ we have bins $b_i = [a_{2i-1},a_{2i}]$ for $i \geq 1$ and $b_0 = [a_0]$.  The first few terms of the sequence are
\begin{equation*} \underbracket{\ 1\ }_{b_0}\ ,\ \underbracket{\    2,\  4\ }_{b_1}\ ,\ \underbracket{\ 6,\ 12\ }_{b_2}\ ,\ \underbracket{\ 18,\ 36\ }_{b_3}\ ,\ \underbracket{\ 54, 108\  }_{b_4}\ ,\ \underbracket{\ 162,\  324\ }_{b_5}\ ,\ \ldots.\end{equation*}

In the case where $m=3$ the triangle ($3$-gonal) sequence begins with the terms
\begin{equation*} \underbracket{\ 1\ }_{b_0}\ ,\ \underbracket{\    2,\  4,\ 6\ }_{b_1}\ ,\ \underbracket{\ 8,\ 16,\ 24\ }_{b_2}\ ,\ \underbracket{\ 32,\ 64,\ 96\ }_{b_3}\ ,\ \underbracket{\ 128,\  256,\ 384\ }_{b_4}\ ,\ \underbracket{\ 512,\  1024,\ 1536\ }_{b_5}\ ,\ \ldots.\end{equation*}
Figure \ref{triangle} gives a visualization of the beginning of the triangle sequence when the integers are placed in the faces of the circumscribed triangles. Moreover, we note that the triangles used need not be equilateral.
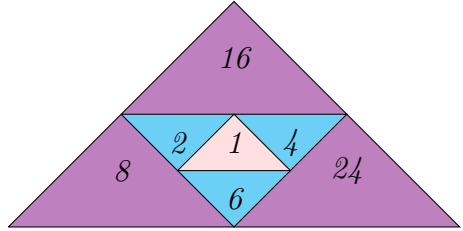
\begin{figure}[h]
\centering
\begin{tikzpicture}[scale=.75]
\draw [fill=violet!50](10,-1) -- (14,3) -- (18,-1) -- (10,-1);
\draw [fill=cyan!50] (12,1) -- (16,1) -- (14,-1) -- (12,1);
\draw [fill=pink!50] (13,0) -- (14,1) -- (15,0) -- (13,0);
\node at (14,.5) {1};
\node at (13,.5) {2};
\node at (15, .5) {4};
\node at (14, -.5) {6};
\node at (12, 0) {8};
\node at (14, 2) {16};
\node at (16, 0) {24};
\end{tikzpicture}
\caption{Beginning of triangle sequence}\label{triangle}
\end{figure}

Also one can observe that the triangle decomposition of $2015$ is given by
\begin{equation}2015=a_{15}+a_{12}+a_{6}+a_{3}+a_{0}=1536 + 384 + 64 + 24 + 6 + 1.\end{equation}
\end{example}

In Section \ref{sec:reccurencerelations} we derive the recurrence relation and explicit closed form expressions for the terms of the the $m$-gonal sequence, which we state below.

\begin{theorem}\label{thm:recurrencevalues} Let $m\geq 1$. If $\{a_n\}_{n=0}^{\infty}$ is the $m$-gonal sequence, then
\begin{equation} a_n\ = \ \begin{cases}
1 &\mbox{{\rm if} $n=0$}\\
2n &\mbox{{\rm if} $1 \leq n \leq m$}\\
 (m+1) a_{n-m} &\mbox{{\rm if} $n > m$}.\end{cases}\end{equation}
Then for $n \geq 1$, with $n=km+r$, $k\geq 0$ and $1\leq r\leq m$
\begin{equation} a_{n}\ = \ 2r(m+1)^k\ .\end{equation}
\end{theorem}

\subsection{Uniqueness of Decomposition}

Notice that for $m\geq 2$, the recurrence given in Theorem \ref{thm:recurrencevalues} is not a positive linear recurrence as the leading coefficients of the first $m$ terms are zero. Therefore past results on positive linear recurrences do not apply to the $m$-gonal sequence; however, we do still obtain unique decomposition.

\begin{theorem}[Uniqueness of decompositions]\label{theorem:unique}
Fix $m\geq 1$. Every positive integer can be written uniquely as a sum of distinct terms from the $m$-gonal sequence, where no two summands are in the same bin.\end{theorem}

A proof of Theorem \ref{theorem:unique} is given in Appendix~\ref{App:A}.

\subsection{Gaussianity}  Previous work with positive linear recurrence sequences proved the number of summands in the decomposition of positive integers converges to a Gaussian (see among others \cite{DDKMMV,MW2}).  The same holds for Kentucky decompositions despite the fact that the Kentucky sequence is not a positive linear recurrence \cite{CFHMN1}, and also for the $m$-gonal sequences.

\begin{theorem}[Gaussian Behavior of Summands]\label{theorem:gaussian}
Let the random variable $Y_n$ denote the number of summands in the (unique) $m$-gonal decomposition of an integer picked at random from $[0, a_{mn+1})$ with uniform probability.\footnote{Using the methods of \cite{BDEMMTTW}, these results can be extended to hold almost surely for a sufficiently large sub-interval of $[0, a_{mn+1})$.} Normalize $Y_n$ to $Y_n' =  (Y_n - \mu_n)/\sigma_n$, where $\mu_n$ and $\sigma_n$ are the mean and variance of $Y_n$ respectively. Then
\begin{align} \mu_n  \ = \   \frac{mn}{m+1}+\frac{1}{2}, \ \ \ \ \ \
\sigma_n^2  \ = \   \frac{mn}{(m+1)^2}+\frac{1}{4}, \end{align} and $Y_n'$ converges in distribution to the standard normal distribution as $n \rightarrow \infty$.
\end{theorem}

The proof of Theorem \ref{theorem:gaussian} is given in Section \ref{sec:gaussian}.

\subsection{Gaps between summands}  Another property studied of positive linear recurrence sequences is the behavior of the gaps between adjacent summands in decompositions, where, in many instances, it has been shown that there is exponential decay in the distribution of gaps, see \cite{BBGILMT,B-AM,BILMT}.\footnote{The proofs involve technical arguments concerning roots of polynomials associated to the recurrence; in many cases one needs to assume all the recurrence coefficients are positive.} Similarly, the Kentucky Sequence displays exponential decay in the distribution of gaps \cite{CFHMN1}. We obtain similar behavior again, though now there is a slight dependence on the residue of gap modulo $m$ (if we split by residue we obtain geometric decay).

Before stating our result we first fix some notation. For the legal $m$-gonal decomposition \begin{equation} z\ = \ a_{\ell_k} + a_{\ell_{k-1}}  + \cdots + a_{\ell_1} \ \ \ {\rm with} \ \ \ \ell_1\ <\ \ell_2 \ < \  \cdots \ < \  \ell_k \end{equation} and $z \in [0,a_{mn+1})$, we define the multiset of gaps as follows: \begin{equation}\text{Gaps}_n(z) \ := \  \{\ell_2-\ell_1, \ell_3 - \ell_2, \dots, \ell_k - \ell_{k-1}\}.\end{equation}  Observe that we do not consider $\ell_1 - 0$, as a gap. However, doing so would not affect the limiting behavior.  For example, notice $z =a_{15}+a_{12}+a_{6}+a_{3}+a_{0}$ contributes three gaps of length 3, and one gap of length 6.

Considering all the gaps between summands in legal $m$-gonal decompositions of all $z \in [0, a_{mn+1})$, we let $P_n(g)$ be the fraction of all these gaps that are of length $g$. That is, $P_n(g)$ is the probability of a gap of length $g$ among legal  $m$-gonal decompositions of $z\in[0,a_{mn+1})$.

\begin{theorem}[Average Gap Measure]\label{thm:gapstheorem}
Let $g=m\alpha+\beta$, where $\alpha\geq 0$ and $0\leq \beta<m$. For $P_n(g)$ as defined above, the limit $P (g) := \lim_{n\to\infty} P_n(g)$ exists, and
\begin{equation}P(g)\ = \ \begin{cases}
\frac{\beta}{m(m+1)}&\mbox{if $\alpha=0$}\\
\frac{m+1-\beta}{(m+1)^{\alpha+1}}&\mbox{if $\alpha> 0$}
.\end{cases}\end{equation}
\end{theorem}

The proof for Theorem \ref{thm:gapstheorem} is given in Section \ref{sec:gaps}.

Via an application of \cite[Theorem 1.1]{DFFHMPP2} we extract a result on individual gaps for the $m$-gonal case. In order to state the theorem, we need the following definitions, as were presented in \cite{DFFHMPP2}, but specialized to the $m$-gonal case. Let $\{a_n\}$ denote the $m$-gonal sequence with its unique decomposition as given in Definition \ref{mDefi}. Let $I_n:= [0,a_{mn+1})$ for all $n>0$ and let $\delta(x-a)$ denotes the Dirac delta functional, assigning a mass of 1 to $x=a$ and 0 otherwise.

\begin{itemize}

\item \emph{Spacing gap measure:} We define the spacing gap measure of a $z \in I_n$ with $k(z)$ summands as
\begin{equation}
\nu_{z,n}(x) \ := \ \frac{1}{k(z)-1}\sum_{j=2}^{k(z)}\delta(x-(\ell_j-\ell_{j-1})).
\end{equation}

\item \emph{Average spacing gap measure:} Note that the total number of gaps for all $z\in I_n$ is
\begin{equation}
N_{\rm gaps}(n) \ :=\ \sum_{z=a_{0}}^{a_{mn+1}-1}(k(z)-1).
\end{equation}
The average spacing gap measure for all $z\in I_n$ is
\begin{align} \nu_n(x) & \ :=\ \frac1{N_{{\rm gaps}}(n)} \sum_{z=a_{0}}^{a_{mn+1}-1} \sum_{j=2}^{k(z)} \delta\left(x - (\ell_j - \ell_{j-1})\right) \nonumber\\
&\  = \ \frac1{N_{{\rm gaps}}(n)} \sum_{z=a_{0}}^{a_{mn+1}-1} \left(k(z)-1\right) \nu_{z,n}(x).
\end{align}
Letting $P_n(g)$ denote the probability of a gap of length $g$ among all gaps from the decompositions of all $z\in I_n$, we have
\begin{equation}
 \nu_n(x) \ = \ \sum_{g=0}^{mn} P_n(g) \delta(x - g).
\end{equation}

\item \emph{Limiting average spacing gap measure, limiting gap probabilities:} If the limits exist, we let
\begin{equation} \label{gaplim}
\nu(x) \ = \ \lim_{n\to\infty} \nu_n(x), \ \ \ \ P(k) \ = \ \lim_{n \to \infty} P_n(k).
\end{equation}

\item \emph{Indicator function for two gaps:}
For $g_1,g_2 \ge 0$
\begin{align}
X_{j_1,j_1+g_1, j_2,j_2+g_2}(n)  & \ := \  \#\left \{
z \in I_n:\begin{subarray} \ a_{j_1}, a_{j_1+g_1}, a_{j_2}, a_{j_2+g_2}\ \text{in}\ z\text{'s\ decomposition,}\\
\text{but\ not\ } a_{j_1+q}, a_{j_2+p}\ \text{for}\ 0<q<g_1, 0<p<g_2
\end{subarray}
\right \}.
\end{align}

\item \emph{Specific gap length probability:} Recall that $P_n(g)$ is the probability
\begin{equation}
P_n(g)\ :=\ \frac{1}{N_{\rm gaps}(n)}\sum_{i=1}^{mn+1-g}X_{i,i+g}(n).
\end{equation}

\end{itemize}

Now can now state the result of the individual gap measure for the $m$-gonal case.

\begin{theorem}[Individual Gap Measure] \label{genthm}
For $z\in I_n$, the individual gap measures $\nu_{z,n}(x)$ converge almost surely in distribution to the limiting gap measure $\nu(x)$.
\end{theorem}
We give a proof of Theorem~\ref{genthm} in Section~\ref{sec:indgaps}.

\subsection{Longest Gap}

Another interesting problem is to determine the distribution of the longest gap between summands as $n\to\infty$. The structure of the legal $m$-gonal decompositions allows us to easily prove the following.

\begin{theorem}[Distribution of the Longest Gap] \label{thm:longestgap} Consider the $m$-gonal sequence $\{a_n\}$. Then as $n\to\infty$ the mean of the longest gap between summands in legal $m$-gonal decompositions of integers in $[a_n, a_{n+1})$ is $m\log_{2}(n/2m) + O_m(1)$, and the variance is $O_m(1)$. \end{theorem}

The proof of Theorem \ref{thm:longestgap} is given in Section \ref{sec:long} and bypasses many of the technical arguments used in \cite{BILMT}. There the authors had to deduce properties of somewhat general associated polynomials; the nature of the legal $m$-gonal decompositions here allows us to immediately convert this problem to a simple generalization of the longest run of heads problem.


\section{Recurrence relations and generating functions}\label{sec:reccurencerelations}

Let $m\geq 1$. We can use the division algorithm to observe that the integer $a_{mk+r}$ is the $r^{\text{th}}$ integer in the bin $b_{k+1}$ for $mk+r \geq 1$. Hence $1\le r\le m$ denotes the location of the integer within its bin.  We let the first bin $b_0$ contain the element $a_0=1$. Then for any $k\geq 0$, we let $b_{k+1}$ denote the set of elements of the $(k+1)^\text{th}$ bin. Namely
\begin{equation}\underbracket{\ a_0\ }_{b_0} \ ,\ \underbracket{\    a_1, a_2, \ldots, a_m \ }_{b_1}\ ,\ \underbracket{\     a_{m+1}, a_{m+2}, \ldots, a_{2m} \ }_{b_2}\ ,\ \ldots\ ,\ \underbracket{\    a_{mk+1}, a_{mk+2}, \ldots, a_{m(k+1)} \ }_{b_{k+1}}\ ,\ \ldots.\end{equation}
We can now begin our work in describing the terms of this sequence.

The following result, which follows immediately from the definition, is used in many of the proofs in this section. We record it here for easy reference.
\begin{definition}\label{lem-omega}
Let $\Omega_n$ denote the integer with summands from each bin $b_0,b_1,b_2,\ldots,b_{n}$.   Then
\begin{equation}\label{omega}
\Omega_n \ =\ \displaystyle\sum_{i=0}^{n}a_{mi}.
\end{equation}
\end{definition}

The first result that makes use of Definition \ref{lem-omega} is given below.

\begin{lemma}\label{polyadd1}
Let $m \geq 1$ and $k\geq 1$.  If $a_{mk+1}$ is the first entry in bin $b_{k+1}$, then $a_{mk+1}=a_{mk}+a_{m(k-1)+1}$.
\end{lemma}

\begin{proof}
We note that since $a_{mk+1}$ and $a_{m(k-1)+1}$ are the first numbers in the bins $b_{k+1}$ and $b_k$, respectively, by Equation \eqref{omega}  we have that
\begin{align}
a_{mk+1}&\ = \ 1+\Omega_k\ = \ 1+ \displaystyle\sum_{i=0}^{k}a_{mi}\label{eq1p}\\
a_{m(k-1)+1}&\ = \ 1+\Omega_{k-1}\ = \  1+\displaystyle\sum_{i=0}^{k-1}a_{mi}.\label{eq2p}
\end{align}
Then Equations \eqref{eq1p} and \eqref{eq2p} yield
\begin{equation}
a_{mk+1}\ = \ 1+ \Omega_k
\ = \ 1+a_{mk}+ \Omega_{k-1}
\ = \ a_{mk}+\left(1+ \Omega_{k-1}\right)
\ = \ a_{mk}+a_{m(k-1)+1},
\end{equation}
as claimed.
\end{proof}

We now prove the more general result.

\begin{lemma}\label{polyadd2}
If $k\geq 0$ and $1\leq r\leq m$, then $a_{mk+r}=r\cdot a_{mk+1}.$
\end{lemma}

\begin{proof}

\noindent

First consider the bin $b_1$.  As $b_0 = [a_0] = 1,$ by construction of the $m$-gonal sequence it is straightforward to determine that $b_1 = [2, 4, \dots, 2m]$ and $a_r = r\cdot a_1$ for all $1 \leq r \leq m.$

We proceed for bins $b_k$ with $k \geq 1$ by induction on $r$, where $1\leq r\leq m$. The basis case when $r=1$ clearly holds.

\noindent
Let $1\leq x\leq m-1$ and assume that for any $1\leq r\leq x$, we have that $a_{mk+r}=r\cdot a_{mk+1}$.
We want to show that $a_{mq+x+1}=(x+1) a_{mk+1}$.
Recall that $a_{mk+x+1}$ is the entry in bin $b_{k+1}$ after $a_{mk+x}$ and by definition $a_{mk+x+1}$ is one more than the largest integer we can create using the elements of bins $b_0, b_1,\ldots, b_{k}$ along with the element $a_{mk+x}$. Using Equation \eqref{omega}, we have that
\begin{align}a_{mk+x+1}\ = \ 1+a_{mk+x}+\Omega_k\ = \ 1+a_{mk+x}+ \displaystyle\sum_{i=0}^{k}a_{mi}.\label{eq5p}\end{align}

Recalling that $1+\Omega_k=a_{mk+1}$ and by the use of the induction hypothesis, Equation \eqref{eq5p} yields
\begin{equation}
a_{mk+x+1}
\ = \ a_{mk+x}+ a_{mk+1}
\ = \ xa_{mk+1}+ a_{mk+1}
\ = \ (x+1)a_{mk+1}
.\end{equation}\end{proof}

We now provide a closed formula for the terms of the $m$-gonal sequence.

\begin{proposition}\label{polythes} 
Let $m\geq 2$, $k\geq 0$,  and $1\leq r\leq m$. Then
$a_{mk+r}=2r(m+1)^k$.  For $m = 1$, $a_i = 2^i.$
\end{proposition}

\begin{proof}

For the case where $m=1$, each of our bins have size 1 and a legal decomposition has distinct summands.  Thus the rule for legal decomposition is precisely a description of writing the positive integers in binary.

We will proceed by induction on $k$, the subscript on the bin, and $r$, the location of $a_{mk+r}$ within the bin considered. The basis case $k=0$ and $r=1$ gives the expected result,
$a_{m\cdot 0+1}=2(1)(m+1)^0=2$.  We now assume that for some $k\geq 0$ and some $r,$ $1\leq r\leq m$, we have
\begin{equation}a_{mk+r}\ = \ 2r(m+1)^k.\end{equation}
We need to show that the following two equations hold:
\begin{align}
a_{mk+r+1}&\ = \ 2(r+1)(m+1)^k\label{ind1}\\
a_{m(k+1)+r}&\ = \ 2r(m+1)^{k+1}.\label{ind2}
\end{align}
Suppose that $1\leq r\leq m-1$. To show Equation \eqref{ind1} holds it suffices to observe that by Lemma~\ref{polyadd2} and our induction hypothesis we have
\begin{align}\label{iteration}
a_{mk+r+1}&\ = \ (r+1)a_{mk+1}\ = \ (r+1)\cdot 2(1)(m+1)^k\ = \ 2(r+1)(m+1)^k.
\end{align}
When $r=m$, we use Lemma \ref{polyadd1} and our induction hypothesis to quickly deduce

\begin{equation}\label{polyadd2goal} a_{mk+m+1} = 2(m+1)^{k+1}.\end{equation}

By iterating \ref{iteration} till $r = m-1,$ we find that \ref{polyadd2goal} holds.  Then \ref{ind2} holds by Lemma ~\ref{polyadd2}.
\end{proof}
The final result gives the recurrence relation stated in Theorem \ref{thm:recurrencevalues}.
\begin{corollary} If $n> m$, then $a_{n}=(m+1)a_{n-m}$.
\end{corollary}
\begin{proof}
Let $n>m$ and write $n=mk+r$, where $k\geq 1$ and $1\leq r\leq m$. By Proposition \ref{polythes}, $a_n=a_{mk+r}=2r(m+1)^k$ and $a_{n-m}=a_{m(k-1)+r}=2r(m+1)^{k-1}$. So it directly follows that $a_n=(m+1)a_{n-m}$.
\end{proof}

\subsection{Counting integers with exactly $k$ summands}

In \cite{KKMW}, Kolo$\breve{{\rm g}}$lu, Kopp, Miller and Wang introduced a very useful combinatorial perspective to attack Zeckendorf decomposition problems by partitioning the integers $z  \in [F_n, F_{n+1})$ into sets based on the number of summands in their Zeckendorf decomposition. We use a similar technique to prove that the distribution of the average number of summands in the $m$-gonal decomposition displays Gaussian behavior.

 let $p_{n,k}$ denote the number of integers in $I_n:=[0,a_{mn+1})$ whose $m$-gonal decomposition contains exactly $k$ summands, where $k \geq 0$. We begin our analysis with the following result.

\begin{proposition}\label{pnkbig}
If $n,k\geq 0$, then \begin{equation} p_{n,k}\ = \ \begin{cases}1&\text{{\rm if} $k=0$}\\m^k{n\choose k}+m^{k-1}{n\choose k-1}&\text{{\rm if} $1\leq k\leq n+1$}\\0&\text{{\rm if} $k>n+1$}.\end{cases} \end{equation}
\end{proposition}

\begin{proof} Let $n,k\geq 0$. Observe that the unique integer in the interval $I_n=[0,a_{mn+1})$ which has zero summands is zero itself. Thus $p_{n,0}=1$. Now if $k$ is larger than the number of available bins, it would be impossible to have $k$ summands as one can draw no more than one summand per bin. Therefore $p_{n,k}=0$, whenever $k>n+1$.

We now show that for $1\leq k\leq n+1$, $p_{n,k}=m^k{n\choose k}+m^{k-1}{n\choose k-1}$. There are two cases to consider:
\begin{enumerate}[label={Case} \arabic*.,itemindent=1em]
\item One of the $k$ summands is chosen from $b_0$.\label{c1}
\item None of the $k$ summands are chosen from $b_0$.\label{c2}
\end{enumerate}

\noindent
\ref{c1} Since one of the $k$ summands is coming from $b_0$ there are $k-1$ available summands to take from the bins $b_1,\ldots, b_n$. The number of ways to select $k-1$ bins from $n$ bins is $n \choose k-1$. As each of the bins $b_1, \dots, b_n$ has exactly $m$ elements and $|b_0|=1$, once the $(k-1)$ bins are selected, the number of ways to select an element from these bins is $m^{k-1}$. Thus the number of $z\in I_n$ which have exactly $k$ summands with one summand coming from bin $b_0$ is $m^{k-1}\binom{n}{k-1}$.\\

\noindent
\ref{c2} We choose $k$ summands from any bin but $b_0$. Using a similar argument as in Case 1, we can see that the total number of ways to select these $k$ summands is $m^{k}{n\choose k}$.

As the two cases are disjoint, we have shown that the total number of integers in the interval $I_n$ with exactly $k$ summands is \begin{equation} p_{n,k} \ = \ m^k{n\choose k}+m^{k-1}{n\choose k-1}.\end{equation}
\end{proof}

We also provide a recursive formula for the value of $p_{n,k}$ as it is used in the proof of Proposition~\ref{gen}.

\begin{proposition}\label{pnkrecursive}
If $0<k<n+1$, then $p_{n,k}=mp_{n-1,k-1}+p_{n-1,k}$.
\end{proposition}

We omit the proof of Proposition \ref{pnkrecursive} as it is a straightforward application of the combinatorial identity ${n \choose k} = {{n-1} \choose k} + {{n-1} \choose {k-1}}$.  With the recursive formula at hand, we now create a generating function for $p_{n,k}$.

\begin{proposition}\label{gen}
Let \begin{equation} F(x,y)\ :=\ \sum_{n,k\geq0}p_{n,k} x^ny^k\end{equation} be the generating function of the $p_{n,k}$'s arising from $m$-gonal decompositions. Then
\begin{equation}\label{cform}
F(x,y) \ = \  \frac{1+y}{1-(my+1)x}.  \end{equation}
\end{proposition}

\begin{proof}
Noting that $p_{n,k} = 0$ if either $n<0$ or $k<0$, using explicit values of $p_{n,k}$ and the recurrence relation from Proposition \ref{pnkrecursive}, after some straightforward algebra we obtain
\begin{equation} F(x,y) \ = \  mxyF(x,y)+xF(x,y)+1+y. \end{equation}
From this, Equation \eqref{cform} follows.
\end{proof}


\section{Gaussian behavior}\label{sec:gaussian}

To motivate this section's main result, we point the reader to the following experimental observations. Taking samples of 200,000 integers from the intervals $[0, 2(4)^{600})$, $[0, 2(5)^{600})$, $[0, 2(6)^{600})$ and $[0, 2(7)^{600})$, in Figure \ref{gaussiangraphs} we provide a histogram for the distribution of the number of summands in the $m$-gonal decomposition of these integers, when $m=3,4,5$ and $6$, respectively. Moreover, Figure \ref{gaussiangraphs} provides the histograms and Gaussian curves (associated to the respective value of $m$ and $n$; the interval is $[0, a_{mn+1})$ so $n=600$ in all experiments). In Table \ref{table:gaussian} we give the values of the predicted means and variances (as computed using Proposition \ref{meanvariance}), as well as the sample means and variances, for each of the cases considered.

\begin{figure}[h]
  \centering
  \caption{Distributions for the number of summands in the $m$-gonal decomposition for a random sample with $n=600$. Unfortunately the image is not
  displaying on the arXiv; see \bburl{http://web.williams.edu/Mathematics/sjmiller/public_html/math/papers/mgon_73.pdf}.  }
  \label{gaussiangraphs}
\end{figure}

\begin{table}[h]
\centering
\begin{tabular}{|c|c|c|c|c|c|}
\hline
Figure&$m$&Predicted Mean&Sample Mean&Predicted Variance&Sample Variance\\\hline\hline
\ref{gaussian m=3}&$3$&$450.50$& $450.49$& $112.75$& $112.34$\\\hline
\ref{gaussian m=4}&$4$&$480.50$&$480.52$&\ $96.25$&\  $95.73$\\\hline
\ref{gaussian m=5}&$5$&$500.50$&$450.49$&\ $83.58$&\ $83.38$\\\hline
\ref{gaussian m=6}&$6$&$514.79$&$514.76$&\ $73.72$&\ $73.64$\\\hline
\end{tabular}
\caption{Predicted means and variances versus sample means and variances for simulation from Figure  \ref{gaussiangraphs}.}
\label{table:gaussian}
\end{table}

From these observations it is expected that for any $m\geq 1$, the distribution of the number of summands in the $m$-gonal decompositions of integers in the interval $I_n$ displays Gaussian behavior. This is in fact the statement of Theorem \ref{theorem:gaussian}.
We begin by proving a technical result and follow it with the formulas for the mean and variance, which make use of some properties associated with the generating function for the $p_{n,k}$'s.

\begin{proposition}\label{coef}
If $g_n(y)$ denotes the coefficient of $x^n$ in $F(x,y)$, then \begin{equation}g_n(y) \ = \ (1+y)(my+1)^n.\end{equation}
\end{proposition}

\begin{proof}

Using the fact that $F(x,y)=\frac{1+y}{1-mxy-x}$ we have by geometric series that \begin{equation}F(x,y)\ = \ \sum_{n=0}^{\infty}(1+y)(my+1)^nx^n.\end{equation}  Thus the coefficient of $x^n$ in $F(x,y)$ is $(1+y)(my+1)^n$.
\end{proof}

We can now use $g_n(y)$ to find the mean and variance for the number of summands for integers $z\in I_n$.

\begin{proposition}\label{meanvariance}
Let $Y_n$ be the number of summands in the $m$-gonal decomposition of a randomly chosen integer in the interval $I_n$, where each integer has an equal probability of being chosen. Let $\mu_n$ and $\sigma_n^2$ denote the mean and variance of $Y_n$. Then
\begin{align}
\mu_n \ = \ \frac{nm}{m+1}+\frac{1}{2}, \ \ \ \ \sigma^2_n \ = \ \frac{nm}{(m+1)^2}+\frac{1}{4}.
\end{align}
\end{proposition}

\begin{proof}
By Propositions 4.7 annd 4.8 in \cite{DDKMMV} the mean and variance of $Y_n$ are
\begin{align}
\mu_n&\ = \ \displaystyle\sum_{i=0}^{n}iP(Y_n=i)\ = \ \displaystyle\sum_{i=0}^{n}i\frac{p_{n,i}}{\mbox{\tiny $\displaystyle\sum_{k=0}^{n}p_{n,k}$}}\ = \ \frac{g'_n(1)}{g_n(1)}, \text{\ and} \label{m1}\\
\sigma^2_n&\ = \ \displaystyle\sum_{i=0}^{n}(i-\mu_n)^2P(Y_n=i)\ = \ \displaystyle\sum_{i=0}^{n}i^2\frac{p_{n,i}}{\mbox{\tiny $\displaystyle\sum_{k=0}^{n}p_{n,k}$}}-\mu^2_n\ = \ \frac{\mbox{\tiny $\frac{d}{dy}$}[yg'_n(y)]|_{y=1}}{g_n(1)}-\mu^2_n.\label{v1}
\end{align}
Our result follows directly from these formulas and the fact that $g_n(y)=(1+y)(my+1)^n$.\end{proof}

Normalize $Y_n$ to $Y_n' = \frac{Y_n - \mu_n}{\sigma_n}$, where $\mu_n$ and $\sigma_n$ are the mean and variance of $Y_n$ respectively, as given in Proposition \ref{meanvariance}.
We are now ready to prove that $Y'_n$ converges in distribution to the standard normal distribution as $n \rightarrow \infty$.

\begin{proof}[Proof of Theorem \ref{theorem:gaussian}]
For convenience we set $r:=\frac{t}{\sigma_n}$ . Since $\sigma_n=\sqrt{\frac{nm}{(m+1)^2}+\frac{1}{4}}$, we know that $r\to 0$ as $n\to \infty$ for any fixed value of $t$. Hence we will expand $e^r$ using its power series expansion. We start with
\begin{align}
M_{Y'_n}(t)&\ = \ \frac{g_n(e^{\frac{t}{\sigma_n}})e^{\frac{-t\mu_n}{\sigma_n}}}{g_n(1)}.\label{eqx}
\end{align}
Taking the logarithm of Equation \eqref{eqx}
\begin{align}
\log (M_{Y'_n}(t))&\ = \ \log [g_n(e^{r})]-\log [g_n(1)]-\frac{t\mu_n}{\sigma_n}.\label{logM}
\end{align}
We proceed using Taylor expansions of the exponential and logarithmic functions to expand the following:
\begin{align}
\log[g_{n}(e^r)]&\ = \ \log(1+e^r)+n\log (me^r+1)\nonumber\\ &\ = \ \log\left(1+\left(1+r+\frac{r^2}{2}\right)\right)+n\log\left(m\left(1+r+\frac{r^2}{2}\right)+1\right)+O(r^3)\nonumber\\
%
&\ = \  \log(2)+\frac{1}{2}\left(r+\frac{r^2}{2}\right)-\frac{1}{8}\left(r+\frac{r^2}{2}\right)^2\nonumber\\
& \hspace{.75in}+n\left[\log(m+1)+\frac{\left(mr+\frac{mr^2}{2}\right)}{m+1}-\frac{\left(mr+\frac{mr^2}{2}\right)^2	 }{2(m+1)^2}\right]+O(r^3)\nonumber\\
&\ = \ \log(2(m+1)^n)+\frac{r}{2}+\frac{r^2}{8}+\frac{nmr}{m+1}+\frac{nmr^2}{2(m+1)^2}+O(r^3).\label{loggny}
\end{align}
From Proposition \ref{coef} we have that $g_n(1)=2(m+1)^n$, hence
\begin{align}\log[2(m+1)^n]\ = \ \log[g_n(1)].\label{loggn1}\end{align}

Substituting Equations \eqref{loggny} and \eqref{loggn1} and the values $\mu_n=\frac{nm}{m+1}+\frac{1}{2}$ and $\sigma_n=\sqrt{\frac{nm}{(m+1)^2}+\frac{1}{4}}$ into Equation~\eqref{logM} yields
\begin{align}
\log (M_{Y'_n}(t))&\ = \ \frac{r}{2}+\frac{r^2}{8}+\frac{nmr}{m+1}+\frac{nmr^2}{2(m+1)^2}-\frac{t\left(\frac{nm}{m+1}+\frac{1}{2}\right)}{\sqrt{\frac{nm}{(m+1)^2}+\frac{1}{4}}}+O(r^3).
\end{align}
After some straightforward algebra we arrive at
\begin{equation}\log(M_{Y'_n}(t))\ = \ \frac{t^2}{2}+o(1);\end{equation}
the moment generating proof of the Central Limit Theorem now yields that the distribution converges to that of the standard normal distribution as $n\to\infty$.
\end{proof}

\section{Average Gap Measure}\label{sec:gaps}
We now turn our attention to our final result in which we determine the behavior of gaps between summands. We begin with some preliminary notation in order to make our approach precise. For a positive integer $z\in I_n=[0,a_{mn+1})$ with $m$-gonal decomposition \begin{equation}z\ = \ a_{\ell_t}+a_{\ell_{t-1}}+\cdots+a_{\ell_1}\end{equation} where
$\ell_1<\ell_2<\cdots<\ell_t$, we define the multiset of gaps of $z$ as
\begin{equation}Gaps_n(z)\ :=\ \{\ell_2-\ell_1,\ell_3-\ell_2,\ldots,\ell_t-\ell_{t-1}\}.\end{equation}

Our result will average over all $z \in [0,a_{mn+1})$ since we are interested in the average gap measure arising from $m-$gonal decompositions.

We follow the methods of [BBGILMT, BILMT]. In order to have a gap of length exactly $g$ in the decomposition of $z$, there must be some index $i$ such that $a_i$ and $a_{i+g}$ occur in $z$'s decomposition, but $a_j$ does not for any $j$ between $i$ and $i+g$. Thus for each $i$ we count how many $z$ have $a_i$ and $a_{i+g}$ but not $a_j$ for $i < j < i+g$; summing this count over $i$ gives the number of occurrences of a gap of length $g$ among all the decompositions of $m$ in our interval of interest.  We want to compute the fraction of the gaps (of length $g$) arising from the decompositions of all $z\in I_n$.  This probability is given by
\begin{align}
P_n(g)&\ :=\ \frac{1}{(\mu_n-1)a_{mn+1}}\displaystyle\sum_{z=0}^{a_{mn+1}-1}\displaystyle\sum_{i=0}^{mn+1-g}X_{i,g}(z),
\end{align}
 where $X_{i,g}(z)$ is the indicator function \footnote{For $1\leq i,g\leq mn$ $X_{i,g}(z)$ denotes whether the decomposition of $z$ has a gap of length $g$ beginning at index $i$. That is, for $z =  \ a_{\ell_t}+a_{\ell_{t-1}}+\cdots+a_{\ell_1},$
\begin{align}
X_{i,g}(z) &\ = \ \begin{cases}1&\mbox{if $ \exists \ j$, $1\leq j\leq t$ with $i=\ell_j$ and $i+g=\ell_{j+1}$}\\
0&\mbox{otherwise}.
\end{cases}\label{indicatorfct}
\end{align}}

If $X_{i,g}(z) = 1$, then there exists a gap between $a_i$ and $a_{i+g}$. Namely $z$ does not contain any of the summands $a_{i+1},\ldots , a_{i+g-1} $ in its legal $m$-gonal decomposition.

We are now ready to prove the result on the exponential decay in the distribution of gaps.  The arguments in the proof of our main result (Theorem \ref{thm:gapstheorem}) are quite straightforward, however a bit tedious.  To simplify our arguments, we write the gap length $g$ as $m\alpha+\beta$, where $\alpha\geq 0$ and $0\leq \beta<m$.

\begin{proof}[Proof of Theorem \ref{thm:gapstheorem}]
Let $g=m\alpha+\beta$, where $\alpha\geq 0$ and $0\leq \beta<m$. We proceed by considering the following 
two cases:
\begin{enumerate}[label={Case} \arabic*.,itemindent=1em]
\item $\alpha=0$,\label{case1}
\item $\alpha> 0$. \label{case2}
\end{enumerate}

\ref{case1} Let $\alpha=0$. Hence $g=\beta$, where $0 <  \beta<m$ and so our gap is less than the size of each bin $b_i$ for $i > 0$ ($\alpha = 0$ and $\beta = 0$ would give us a gap of length 0 which is not $m-$gon legal). We first consider gaps of length $g = \beta$ beginning at index 0.  If $a_i=a_0$ then the only way to have a gap of length $g=\beta$ is if $a_{i+1} = a_\beta.$  Now we are counting integers with $m-$gonal decompositions of the form $a_{\ell_t} + \cdots + a_{\ell_3} + a_\beta + a_0$.  The number of $z\in I_n$ with summands $a_{\ell_3}, \dots, a_{\ell_t}$ coming from bins $b_2, b_3,\ldots, b_n$ is $(m+1)^{n-1}$.

If $a_i=a_{mk+r}$, $k\geq 0$ and $1\leq r\leq m$, then $a_{i+g}=a_{mk+r+\beta}$. Notice that the case $r+\beta\leq m$ cannot occur as this would force $\{a_i, a_{i+g}\} \subset b_{k+1}$, which leads to a decomposition that is not $m-$gon legal as only one summand can be taken per bin. Now if $r+\beta >m$, then $a_{i+g}\in b_{k+2}$

Notice that in this case $a_i = a_{mk+r}$ is one of the largest $\beta$ entries in bin $b_{k+1}$ and thus there are $\beta$ many choices for $r$, namely $r \in \{m-\beta + 1, m - \beta + 2, \dots, m \}$.

Now we need to count the number of $z\in I_n\setminus \{0\} = [1, a_{mn+1})$ which have summands $a_i\in b_{k+1}$ and $a_{i+g}\in b_{k+2}$. We must have $0 \leq k \leq n-2$.

As we have already used bins $b_{k+1}$ and $b_{k+2}$, it follows from a straightforward combinatorial counting argument that the total number of integers $z\in I_n$ that can be created with summands $a_i\in b_{k+1}$ and $a_{i+g}\in b_{k+2}$ (with no summands in between) is given by
\begin{align}
2\beta(m+1)^{n-2}\label{eq1},
\end{align}
where the factor of $\beta$ comes from the $\beta$ possible choices of $a_i$ within the bin $b_{k+1}$. As we can vary $k$ from $0$ to $(n-2)$, we find that the total number of integer $z\in I_n$ which contribute a gap between $a_i (i\neq 0)$ and $a_{i+g}$ (assuming that $r+\beta>m)$ is given by
\begin{align}
2(n-1)\beta(m+1)^{n-2}\label{eq2}.
\end{align}

Observe that Equation \eqref{eq2} does not account for the case when $i=0$, which adds an extra factor of $(m+1)^{n-1}$. Therefore
\begin{align}
\displaystyle\sum_{z=0}^{a_{mn+1}-1}\displaystyle\sum_{i=0}^{mn+1-g}X_{i,g}(z)&\ = \ 2\beta(n-1)(m+1)^{n-2}+(m+1)^{n-1}.\label{num}
\end{align}
Using Proposition \ref{meanvariance} and Proposition \ref{polythes} we have that
\begin{align}
(\mu_n-1)a_{mn+1}&\ = \ \left(\frac{mn}{m+1}-\frac{1}{2}\right)\left(2(m+1)^n\right)\ = \ (m+1)^{n-1}(2mn-(m+1)).\label{denom}
\end{align}
By Equations \eqref{num} and \eqref{denom}, for $g=\beta$, with $0\leq \beta<m$, we have that
\begin{align}
P_n(g)
&\ = \ \frac{2\beta(n-1)}{(m+1)(2mn-m-1)} + \frac{1}{2mn-m-1}.\label{png}
\end{align}
Now recall that $P(g)=\lim_{n\to\infty}P_n(g)$, so by letting $n\to\infty$ in Equation \eqref{png} we have that
\begin{align}
P(g)\ = \ \frac{\beta}{m(m+1)},
\end{align}
whenever $g=\beta$ and $0\leq \beta<m$. This completes \ref{case1}\\

\noindent \ref{case2} Let $g=m\alpha+\beta $, where $\alpha\geq 1$ and $0\leq\beta<m$. First consider when $a_i=a_0$. If $\beta = 0$, then $a_{i + g} = a_{m\alpha} \in b_\alpha$.  Otherwise, when $0 < \beta < m,$ $a_{i+g} = a_{m\alpha + \beta} \in b_{\alpha + 1}$.  In the case of the former, the number of $z\in I_n$ with summands coming from bins $b_{\alpha+1}, b_{\alpha+2},\ldots, b_n$ is $(m+1)^{n-\alpha}$.  In the latter case,  the number of $z\in I_n$ with summands coming from bins $b_{\alpha+2}, b_{\alpha+3},\ldots, b_n$ is $(m+1)^{n-(\alpha+1)}$.\\

Now if $a_i=a_{mk+r}$, with $k\geq 0$ and $1\leq r\leq m$, then $a_{i+g}=a_{m(k+\alpha)+r+\beta}$. Hence $a_i\in b_{k+1}$ and $a_{i+g}\in b_{k+\alpha+1}$ whenever $1\leq r+\beta\leq m$, or $a_{i+g}\in b_{k+\alpha+2}$ whenever $m<r+\beta<2m$. Hence we consider the following subcases:

\begin{enumerate}[label={Subcase} \arabic*.,itemindent=1em]

\item Let $1\leq r+\beta\leq m$.\label{subcase1}

\item Let $m< r+\beta< 2m$.\label{subcase2}
\end{enumerate}

\ref{subcase1} Let $1\leq r+\beta\leq m$. In this case $a_i\in b_{k+1}$ and $a_{i+g}\in b_{k+\alpha+1}$. Notice that in this case $a_i$ must be one of the smallest $m-\beta$ entries in bin $b_{k+1}$.

Namely $r=1,2,\ldots,m-\beta$.

Now we need to count the number of $z\in I_n$ which have summands $a_i\in b_{k+1}$ and $a_{i+g}\in b_{k+\alpha+1}$ (and no summands in between). For the decomposition to only have summands from bins $b_0, \dots, b_{k+1}, b_{k + \alpha + 1}, \dots, b_n,$ we must have $0\leq k\leq n-(\alpha+1).$

In order to have a gap created by $a_i\in b_{k+1}$ and $a_{i+g}\in b_{k+\alpha +1}$, there must be no summands taken from $b_j$, where $k+1<j<k+\alpha+1$. Again using a straightforward combinatorial counting argument, the total number of integers $z\in I_n$ which have summands $a_i\in b_{k+1}$ and $a_{i+g}\in b_{k+\alpha+1}$ (with no summands in between) is given by
\begin{align}
2(m-\beta)(m+1)^{n-\alpha-1},\label{eq13}
\end{align}
where the factor of $m-\beta$ comes from the $m-\beta$ possible choices of $a_i$ within the bin $b_{k+1}$.

As we can vary $k$ from $0$ to $n-(\alpha+1)$ we find that the total number of integers $z\in I_n$ which contribute a gap between $a_i$ ($i\neq 0$) and $a_{i+g}$ in this case is
\begin{align}
2(n-\alpha)(m-\beta)(m+1)^{n-\alpha-1}.\label{eq33}
\end{align}

\noindent
\ref{subcase2} Let $m< r+\beta< 2m$. In this case $a_i\in b_{k+1}$ and $a_{i+g}\in b_{k+\alpha+2}$. Notice that in this case $a_i$ can be any of the largest $\beta$ entries in bin $b_{k+1}$ so $a_i\in b_{k+1}$ and $a_{i+g}\in b_{k+\alpha+2}$. Namely $r=m+1-\beta,m+2-\beta,\ldots,m$.

Using the same reasoning as in Subcase 1, we determine that the total number of integers meeting the conditions is
\begin{align}
2\beta(n-\alpha-1)(m+1)^{n-\alpha-2}.\label{eq34}
\end{align}
This completes \ref{subcase2}

We still need to account for the number of integers $z\in I_n$ which contribute a gap of length $g=m\alpha+\beta$ ($\alpha\geq 1$ and $0\leq\beta<m$) beginning at index $a_0$. Recall we previously computed this quantity to be $(m+1)^{n-\alpha-1}$ when $\beta > 0$ and the quantity is $(m+1)^{n-\alpha}$ when $\beta = 0$.

Therefore we need to sum the values of Equations \eqref{eq33}, \eqref{eq34} along with $(m+1)^{n-\alpha-1}$ when $\beta > 0$ to get that
\begin{align}
\displaystyle\sum_{z=0}^{a_{mn+1}-1}\displaystyle\sum_{i=0}^{mn+1-g}X_{i,g}(z)&=
2(n-\alpha)(m-\beta)(m+1)^{n-\alpha-1}\nonumber\\
& \qquad  \qquad +2\beta(n-\alpha-1)(m+1)^{n-\alpha-2}+(m+1)^{n-\alpha-1}.\label{num3}
\end{align}
By Equations \eqref{num3} and \eqref{denom}, for $g=m\alpha+\beta$, with $\alpha\geq 1$ and $0<\beta<m$, we have that
\begin{align}
P_n(g)&\ = \ \frac{2(n-\alpha)(m-\beta)}{(m+1)^\alpha(2mn-m-1)}+\frac{2\beta(n-\alpha-1)}{(m+1)^{\alpha+1}(2mn-m-1)}\nonumber \\
& \  \qquad  \qquad +\frac{1}{(m+1)^\alpha(2mn-m-1)}.\label{png3}
\end{align}
Now recall that $P(g)=\lim_{n\to\infty}P_n(g)$, so by letting $n\to\infty$ in Equation \eqref{png3} we have that
\begin{align}
P(g)&\ = \ \frac{2(m-\beta)}{(m+1)^\alpha(2m)}+\frac{2\beta}{(m+1)^{\alpha+1}(2m)}
\ = \ \frac{m+1-\beta}{(m+1)^{\alpha+1}},
\end{align}
whenever $g=m\alpha+\beta$, $\alpha\geq 1$ and $0<\beta<m$.

Now for $\beta=0$ we do not need to consider when $\alpha = 0$ as this would give us a gap $g = 0.$   Also, as $1 \leq r \leq m$, we only meet the conditions of \ref{subcase1}  Thus  for $\alpha > 0 $ and $\beta = 0$ we need to sum the values of Equations \eqref{eq33} along with $(m+1)^{n-\alpha}$  to get
\begin{align}
\displaystyle\sum_{z=0}^{a_{mn+1}-1}\displaystyle\sum_{i=0}^{mn+1-g}X_{i,g}(z)&=2(n-\alpha)m(m+1)^{n-\alpha-1} + (m+1)^{n-\alpha}.\label{num30}
\end{align}
By Equations \eqref{num30} and \eqref{denom}, for $g=m\alpha$, with $\alpha\geq 1$, we have that
\begin{align}
P_n(g)&\ = \  \frac{2(n-\alpha)m}{(m+1)^{\alpha}(2mn-m-1)} + \frac{1}{(m+1)^{\alpha-1}(2mn-m-1)}.\label{png30}
\end{align}
Now recall that $P(g)=\lim_{n\to\infty}P_n(g)$, so by letting $n\to\infty$ in Equation \eqref{png30} we have that
\begin{align}
P(g)&\ = \frac{1}{(m+1)^{\alpha}} = \frac{m+1}{(m+1)^{\alpha+1}}
\end{align}
whenever $g=m\alpha + \beta $, $\alpha\geq 1$ and $\beta = 0$. This completes the proof.
\end{proof}


\section{Individual Gap Measure}\label{sec:indgaps}
In this section, we prove Theorem~\ref{genthm}, by checking that the conditions given in \cite[Theorem 1.1]{DFFHMPP2} are satisfied in the $m$-gonal case. We restate this theorem below for ease of reference.

\begin{theorem}\label{genthm2}\cite{DFFHMPP2} For $z\in I_n$, the individual gap measures $\nu_{z,n}(x)$ converge almost surely in distribution  to the average gap measure $\nu(x)$ if the following hold.

\begin{enumerate}
\item The number of summands for decompositions of $z \in I_n$ converges to a Gaussian with mean $\mu_n = \cm n+O(1)$ and variance $\sigma_n^2 = \cv n+O(1)$, for constants $\cm, \cv>0$, and $k(z) \ll n$ for all $z \in I_n$.

\item We have the following, with $\lim_{n\to\infty} \sum_{g_1, g_2} {\rm error}(n,g_1,g_2) = 0$:
\begin{align}
\frac{2}{|I_n|\mu_n^2}\sum_{j_1<j_2}X_{j_1,j_1+g_1,j_2,j_2+g_2}(n) \ = \  P(g_1)P(g_2)+\text{{\rm error}}(n,g_1,g_2).
\end{align}

\item The limits in Equation~(\ref{gaplim}) exist.
\end{enumerate}
\end{theorem}

We note that the above theorem is more general than we need, and in our particular case our interval of interest is $I_n = [0,a_{mn+1})$. Now observe that Proposition~\ref{meanvariance} and Theorem~\ref{theorem:gaussian} ensure that the first criterion is met, and $k(z)$ is clearly at most $mn+1$ and thus $k(z) \ll n$. In addition, the exponential decay seen in Theorem~\ref{thm:gapstheorem} shows that Condition (3) is met. It remains to show that Condition (2) of Theorem~\ref{genthm2} holds.

\begin{proposition}
We have that
\begin{align}
\frac{2}{|I_n|\mu_n^2}\sum_{j_1<j_2}X_{j_1,j_1+g_1,j_2,j_2+g_2}(n) \ = \  P(g_1)P(g_2)+\text{{\rm error}}(n,g_1,g_2)
\end{align}
and the sum of the error over all pairs $(g_1,g_2)$ goes to zero as $n\to\infty$.
\end{proposition}

\begin{proof} Let $g_1=\alpha_1m+\beta_1$, $g_2=\alpha_2m+\beta_2$, $j_1 = k_1m+r_1$, and $j_2=k_2m+r_2$ where $0\leq \beta_1,\beta_2< m$, $1\leq k_1,k_2\leq m$ and $k_1<k_2$. Thus $a_{j_1}$ and $a_{j_2}$ are in bins $(k_1+1)$ and $(k_2+1)$ respectively. There are a number of cases to consider when determining $\sum_{j_1<j_2}X_{j_1,j_1+g_1,j_2,j_2+g_2}(n)$, depending on whether or not $\alpha_1=0$ or $\alpha_2=0$. We include only the case when both $\alpha_1,\alpha_2\geq 1$, as the other cases are similar.

There are several subcases to consider. We will first consider the four subcases which contribute to the main term and then bound the remaining cases. In all of the cases contributing to the main term we have $j_1+g_1\neq j_2$ and $j_1\neq 0$ and thus we will suppose these conditions hold below.\\

\textbf{Subcase 1:} Let $1\leq r_1+\beta_1\leq m$ and $1\leq r_2+\beta_2\leq m$.
First we determine the possible values of $k_1$ and $k_2$. In this case, the gap from $g_1$ spans $\alpha_1+1$ bins and the gap from $g_2$ spans $\alpha_2+1$ bins. Thus $0\leq k_1\leq (n-\alpha_1-\alpha_2 -3)$ and $(k_1+\alpha_1+2)\leq k_2\leq (n-\alpha_2-1)$ and the number of choices for $k_1$ and $k_2$ is $n^2/2+O(n)$. Because of the restrictions of $1\leq r_1+\beta_1\leq m$ and $1\leq r_2+\beta_2\leq m$, within each bin there are $(m-\beta_1)$ choices for where to place $a_{j_1}$ and $(m-\beta_2)$ choices for where to place $a_{j_2}$. Lastly, we determine the number of ways to choose the remaining elements for the decomposition. Because we are spanning $\alpha_1+1$ bins for the gap from $g_1$ and $\alpha_2+1$ bins for the gap from $g_2$, using a straight forward combinatorial counting argument, the number of integers that can be made using what remains is $2(m+1)^{n-\alpha_1-\alpha_2-2}$. Thus the total number of integers that can be made in this case is

\begin{eqnarray} & & 2(m-\beta_1)(m-\beta_2)(m+1)^{n-\alpha_1-\alpha_2-2}(n^2/2+O(n)) \nonumber\\
& & \ \ \ \ \ \ \  =\  (m-\beta_1)(m-\beta_2)(m+1)^{n-\alpha_1-\alpha_2-2}(n^2+O(n)). \end{eqnarray}
Through similar arguments we can obtain the remaining three subcases that contribute to the main term.

\textbf{Subcase 2:} Let $1\leq r_1+\beta_1\leq m$ and $m< r_2+\beta_2< 2m$. Then the number of integers that can be made in this case is
\begin{align}
(m-\beta_1)\beta_2(m+1)^{n-\alpha_1-\alpha_2-3}(n^2+O(n)).
\end{align}

\textbf{Subcase 3:} Let $m< r_1+\beta_1< 2m$ and $1\leq r_2+\beta_2\leq m$. Then the number of integers that can be made in this case is
\begin{align}
(m-\beta_2)\beta_1(m+1)^{n-\alpha_1-\alpha_2-3}(n^2+O(n)).
\end{align}

\textbf{Subcase 4:} Let $m< r_1+\beta_1< 2m$ and $m< r_2+\beta_2< 2m$. Then the number of integers that can be made in this case is
 \begin{align}
\beta_1\beta_2(m+1)^{n-\alpha_1-\alpha_2-4}(n^2+O(n)).
\end{align}

The remaining cases occur when $j_1+g_1=j_2$ or $j_1=0$. In these cases, the number of choices for $k_1$ and $k_2$ is on the order of $n$ instead of $n^2$ and thus the number of integers that can be made in these cases is $O(n(m+1)^{n-\alpha_1-\alpha_2})$. Combining all cases, we have
\begin{align}
\sum_{j_1<j_2}X_{j_1,j_1+g_1,j_2,j_2+g_2}(n) \ = \ &\  n^2(m-\beta_1)(m-\beta_2)(m+1)^{n-\alpha_1-\alpha_2-2}\nonumber \\
&\qquad+ \  n^2(m-\beta_1)\beta_2(m+1)^{n-\alpha_1-\alpha_2-3}\nonumber\\
&\qquad+ \  n^2(m-\beta_2)\beta_1(m+1)^{n-\alpha_1-\alpha_2-3}\nonumber\\
&\qquad+ \  n^2\beta_1\beta_2(m+1)^{n-\alpha_1-\alpha_2-4} + O(n(m+1)^{n-\alpha_1-\alpha_2}).
\end{align}
Next, recall that by Proposition~\ref{meanvariance}  $\mu_n = \frac{nm}{m+1}+\frac{1}{2}$. In addition, $|I_n| = a_{mn+1} = 2(m+1)^n$. Thus in our case we have
\begin{align}
&\frac{2}{|I_n|\mu_n^2}\sum_{j_1<j_2}X_{j_1,j_1+g_1,j_2,j_2+g_2}(n)\nonumber\\
\ = \ \qquad & \frac{1}{(m+1)^n\left(\frac{nm}{m+1}+\frac{1}{2}\right)^2}\sum_{j_1<j_2}X_{j_1,j_1+g_1,j_2,j_2+g_2}(n)\nonumber\\
\ = \ \qquad&\  \left(\frac{1}{(m+1)^n\left(\frac{nm}{m+1}+\frac{1}{2}\right)^2}\right)\bigg(n^2(m-\beta_1)(m-\beta_2)(m+1)^{n-\alpha_1-\alpha_2-2}\nonumber\\
&\qquad+ \  n^2(m-\beta_1)\beta_2(m+1)^{n-\alpha_1-\alpha_2-3}\nonumber\\
&\qquad+ \  n^2(m-\beta_2)\beta_1(m+1)^{n-\alpha_1-\alpha_2-3}\nonumber\\
&\qquad+ \  n^2\beta_1\beta_2(m+1)^{n-\alpha_1-\alpha_2-4} + O(n(m+1)^{n-\alpha_1-\alpha_2})\bigg)\nonumber\\
\ = \ \qquad & \left(\frac{1}{(m+1)^{\alpha_1+\alpha_2 + 2}m^2n^2 + O(n)}\right)
\bigg(n^2(m-\beta_1)(m-\beta_2)(m+1)^{2}\nonumber\\
&\qquad+ \  n^2(m-\beta_1)\beta_2(m+1)
\ + \  n^2(m-\beta_2)\beta_1(m+1)
\ + \  n^2\beta_1\beta_2\bigg) \nonumber\\ &\qquad + \ O\left(\frac{1}{n(m+1)^{\alpha_1+\alpha_2}}\right). \end{align}
Taking the limit as $n\rightarrow\infty$ and rearranging we obtain \begin{equation} \frac{m^2(m+1-\beta_1)(m+1-\beta_2)}{m^2(m+1)^{\alpha_1+\alpha_2+2}} \ = \ P(g_1)P(g_2).
\end{equation}
The fact that the error term decays exponentially in $g_1$ and $g_2$ ensures that the error summed over all $g_1$ and $g_2$ goes to zero.
\end{proof}


\section{Longest Gap}\label{sec:long}
Using the techniques introduced by Bower, Insoft, Li, Miller and Tosteson in \cite{BILMT}, we can determine the mean and variance of the distribution of the longest gap between summands in the decomposition of integers in $[a_n, a_{n+1})$ as $n\to\infty$. There are no obstructions to using those methods; however, there are some book-keeping issues due to the nature of our legal $m$-gonal decomposition. Specifically, we have to worry a little about the residue of the longest gap modulo $m$. This is a minor issue, as with probability 1 the longest gap is much larger than $m$ and thus we will not have two items in the same bin. Rather than going through the technical argument, we instead give a very short proof that captures the correct main term of the mean of the longest gap,  which grows linearly with $\log n$; our error is at the level of the constant term for the mean. We are able to handle the variance similarly, and similarly compute that up to an error that is $O_m(1)$.

\begin{proof}[Proof of Theorem \ref{thm:longestgap}] The proof follows immediately from results on the longest run of heads in tosses of a fair coin; we sketch the details below. If a coin has probability $p$ of heads and $q=1-p$ of tails, the expected longest run of heads and its variance in $n$ tosses is \be\label{eq:schillingformula} \mu_n \ = \ \log_{1/p}(nq) - \frac{\gamma}{\log p} - \frac12 + r_1(n) + \epsilon_1(n), \ \ \ \sigma_n^2 \ = \ \frac{\pi^2}{6 \log^2 p} + \frac1{12} + r_2(n) + \epsilon_2(n); \ee here $\gamma$ is Euler's constant, the $r_i(n)$ are at most .000016, and the $\epsilon_i(n)$ tend to zero as $n\to\infty$. Note the variance is bounded independently of $n$ (by essentially 3.5); see \cite{Sch} for a proof.

Note that for legal $m-$gonal decompositions we either have an element in a bin, or we do not. As all decompositions are equally likely, we see that these expansions are equivalent to flipping a coin with probability 1/2 for each bin, and choosing exactly one of the $m$ possible summands in that bin if we have a tail. As the probability that the longest gap is at the very beginning or very end of a sequence of coin tosses is negligible, we can ignore the fact that the first bin has size 1 and that we may only use part of the last bin if $n+1$ is not a multiple of $m$.  Thus gaps between bins used in a decomposition correspond to strings of consecutive heads.

As our integers lie in $[a_n, a_{n+1})$, we have $\lfloor n/m\rfloor + O(1) = n/m + O(1)$ bins (again, we ignore the presence or absence of the initial bin of length one or a partial bin at the end). We now invoke the results on the length of the longest run for tosses of a fair coin. For us, this translates not to a result on the longest gap between summands, but to a result on the longest number of \emph{bins} between summands. It is trivial to pass from this to our desired result, as all we must do is multiply by $m$ (the error will be at most $O(m)$ coming from the location of where the summands are in the two bins). This completes the proof of our claim on the mean; the variance follows similarly.
\end{proof}

\appendix

\section{Proof of Theorem \ref{theorem:unique}}\label{App:A}
\begin{proof}[Proof of Theorem \ref{theorem:unique}]
Our proof is constructive. We build our sequence by only adjoining terms that ensure that we can {\em uniquely} decompose a number while never using more than one summand from the same bin. For a fixed $m\geq 1$ the sequence begins:
\begin{equation}\underbracket{\ 1\ }_{{b}_0}\ ,\ \underbracket{\    2,\  4,\ 6,\ \ldots,\ 2m \ }_{b_1}\ ,\ \underbracket{\      2(m+1),\ 4(m+1),\ 6(m+1),\ \ldots,2m(m+1) \ }_{b_2}\ ,\ \ldots\end{equation}
Note we would not adjoin 7 because then there would be two legal $m$-gonal decompositions for 7, one using $7=7$ and the other being $7=6+1$. The next number in the sequence must be the smallest integer which cannot be legally decomposed using the current terms of the sequence.

We can now proceed with proof by induction. Note that the integers $1,2,3,\ldots,2m$ have unique decompositions as they are either in the sequence or are the sum of an even number from bin $b_1$ plus the $1$ from bin $b_0$. The sequence continues:
\begin{eqnarray}  \ldots  , \underbracket{\ a_{m(n-2)+1}, a_{m(n-2)+2}, \ldots  , a_{m(n-1)} \ }_{b_{n-1}}  ,  \underbracket{\ a_{m(n-1)+1},    \ldots   ,  a_{mn} \ }_{b_n}  ,  \underbracket{\ a_{mn+1},  \ldots   ,  a_{m(n+1)} \ }_{b_{n+1}}  , \ldots\end{eqnarray}
By induction we assume that there exists a unique decomposition for all integers $z\leq a_{mn}+\Omega_{n-1}$, where $\Omega_{n-1}$ is the maximum integer that can be legally decomposed using terms in the set $\{a_0,a_1, a_2, a_3,\dots,a_{m(n-1)}\}$.

By construction we have
\begin{align}
\Omega_n\ = \ a_{mn}+\Omega_{n-1}\ = \ a_{mn}+a_{m(n-1)+1}-1.\nonumber \end{align}

Let $x$ be the maximum integer that can be legally decomposed using terms in the set  $\{a_1, a_2, a_3,\dots,a_{m(n-1)}\}$. Note $x = a_{m(n-1)+1}-1$ as this is why we include $a_{m(n-1)+1}$ in the sequence.

\bigskip
\noindent{\bf Claim:} $a_{mn+1}=\Omega_n+1$ and this decomposition is unique.
\smallskip

By induction we know that $\Omega_n$ was the largest value that we could legally decompose using only terms in $\{a_0,a_1, a_2, \dots, a_{mn}\}$. Hence we choose $\Omega_n+1$ as $a_{mn+1}$ and $\Omega_n+1$ has a unique decomposition.

\bigskip
\noindent{\bf Claim:} All $N \in {[\Omega_n+1, \Omega_n+1+x]=[a_{mn+1},a_{mn+1}+x]}$ have a unique decomposition.
\smallskip

We can legally and uniquely decompose the integers $1,2,3,\ldots,x$ using elements in the set $\{a_0,a_1, a_2,\dots,a_{m(n-1)}\}$. Adding $a_{mn+1}$ to the decomposition of any of these integers would still yield a legal $m$-gonal decomposition since $a_{mn+1}$ is not in any of the bins $b_0,b_1, b_2, \dots, b_{n-1}$. The uniqueness of these decompositions follows from the fact that if $a_{mn+1}$ was not included as a summand, then the decomposition does not yield a number big enough to exceed $\Omega_n+1$.

\bigskip
\noindent{\bf Claim:} $a_{mn+2}=\Omega_n+1+x+1=a_{mn+1}+x+1$ and this decomposition is unique.
\smallskip

By construction the largest integer that can be legally decomposed using terms  $\{a_0$, $a_1$, $a_2$, $\dots$, $a_{mn+1}\}$ is
$\Omega_n + 1 + x$.

\bigskip
\noindent{\bf Claim:} All $N \in {[a_{mn+2}, a_{mn+2}+x]}$ have a unique decomposition.
\smallskip

First note that the decomposition exists as we can legally and uniquely construct $a_{mn+2}+v$, where $0\leq v\leq x$.
For uniqueness, we note that if we do not use $a_{mn+2}$, then the summation would be too  small.

\bigskip
\noindent{\bf Claim:} $a_{mn+2}+x$ is the largest integer that can be legally decomposed using terms  $\{a_0,a_1, a_2$, $\dots$, $a_{mn+2}\}$.
\smallskip

This follows from construction.
\end{proof}

%

\ \\

\end{document}